\documentclass[english, 11pt]{article}

\usepackage[a4paper,top=3cm, bottom=3cm, left=2.75cm, right=2.75cm]{geometry}%margin=3cm
\usepackage{fourier}
\usepackage[T1]{fontenc}
\usepackage[utf8]{inputenc}
\usepackage{babel}
\usepackage{microtype}
\usepackage{amsmath,amsthm,amssymb}
\usepackage{tikz}
\usepackage[round]{natbib}

\usepackage{authblk}\usepackage[colorlinks=true,citecolor=blue,linkcolor=blue,urlcolor=red]{hyperref}

\newcommand{\conv}{\mathop{\mathrm{conv}}}
\newcommand{\bem}[1]{\emph{\textbf{#1}\/}}

\theoremstyle{plain} % just in case the style had changed
\newtheorem{theorem}{Theorem}
\newcommand{\thistheoremname}{}
\newtheorem*{genericthm*}{\thistheoremname}
\newenvironment{namedthm*}[1]
  {\renewcommand{\thistheoremname}{#1}%
   \begin{genericthm*}}
  {\end{genericthm*}}

\begin{document}
\title{\large\textbf{An elementary direct proof that the Knaster-Kuratowski-Mazurkiewicz lemma implies Sperner's lemma}}
\author{\textsc\normalsize Mark Voorneveld}
\affil{\small Department of Economics, Stockholm School of Economics, Box 6501, 113 83 Stockholm, Sweden, \url{mark.voorneveld@hhs.se}}

\date{\small \today}
\maketitle

\begin{abstract}
Three central results in economic theory --- Brouwer's fixed-point theorem, Sperner's lemma, and the Knaster-Kuratowski-Mazurkiewicz (KKM) lemma --- are known to be equivalent. In almost all cases, elementary direct proofs of one of these results using any of the others are easily found in the literature. This seems not to be the case for the claim that the KKM lemma implies Sperner's lemma. The goal of this note is to provide such a proof.


\end{abstract}

\bigskip
\begin{center}
To appear in \textit{Economics Letters}, \url{https://doi.org/10.1016/j.econlet.2017.06.013}.
\end{center}

\newpage
\section{Introduction}

Three central results for economic theory are known to be equivalent: Brouwer's \citeyearpar{Brouwer1911} fixed-point theorem, Sperner's \citeyearpar{Sperner1928} lemma, and the Knaster-Kuratowski-Mazurkiewicz (KKM, \citeyear{KKM1929}) lemma. Shortly after Sperner's publication, \citet[\S 3--4]{KKM1929} showed that Sperner's lemma implies the KKM lemma and that the latter implies Brouwer's fixed-point theorem. Much later, \citet[Thm.~1]{Yoseloff1974} proved that Brouwer's fixed-point theorem implies Sperner's lemma, thereby closing the cycle of implications that makes these results equivalent.

So we have an indirect proof that KKM implies Sperner: KKM implies Brouwer, Brouwer implies Sperner. But my search for an elementary, direct proof came up empty. After recalling relevant definitions, I provide such a proof below. For short direct proofs that Brouwer implies KKM and that Sperner implies Brouwer, see, e.g., \citet[p. 44 and 28]{Border1985}.

\section{Definitions and notation}

Throughout the note, sets lie in $\mathbb{R}^n$. Denote its standard basis vectors by $e_1, \ldots, e_n$. An $m$-\emph{simplex\/} in $\mathbb{R}^n$ is the convex hull
\[
\textstyle S = \conv \{a_1, \ldots, a_{m+1}\} = \left\{\sum_{i=1}^{m+1} \lambda_i a_i \in \mathbb{R}^n: \lambda_1, \ldots, \lambda_{m+1} \geq 0, \sum_i \lambda_i = 1\right\}
\]
of $m+1$ affinely independent vectors $a_1, \ldots, a_{m+1}$ in $\mathbb{R}^n$, its \emph{vertices\/}. Affine independence means that the only scalars for which $\sum_{i=1}^{m+1} \lambda_i a_i = (0, \ldots, 0)$ and $\sum_{i=1}^{m+1} \lambda_i = 0$ are $\lambda_1 = \cdots = \lambda_{m+1} = 0$. For instance, having $n$ vertices, the \emph{unit simplex\/}
\[
\textstyle \Delta = \left\{x \in \mathbb{R}^n: x_1, \ldots, x_n \geq 0, \sum_i x_i = 1\right\} = \conv \{e_1, \ldots, e_n\}
\]
is an $(n-1)$-simplex. Affine independence implies that each element of $S$ has a \emph{unique representation} as a convex combination $\sum_i \lambda_i a_i$ of its vertices. The scalars $\lambda_i$ are called \emph{barycentric coordinates\/} or \emph{weights\/}. The vertices of $S$ are said to \emph{span\/} $S$. A \emph{face\/} of a simplex is a simplex spanned by a subset of its vertices. Using the set of all its vertices, $S$ is a face of itself.

A \emph{subdivision\/} of $\Delta$ is a finite collection of smaller $(n-1)$-simplices whose union is $\Delta$ and where the intersection of any two such smaller simplices is empty or a face of both. A \emph{Sperner labeling\/} assigns a label $1, \ldots, n$ to each vertex of the simplices in the subdivision. The label $\ell(x)$ of each vertex $x$ is chosen among its positive coordinates: $\ell(x) \in \{i: x_i > 0\}$. A simplex in the subdivision is \emph{completely labeled\/} if its set of vertices has all $n$ distinct labels.

The left panel of Figure \ref{fig: simplex} shows a subdivision of the unit simplex in $\mathbb{R}^3$; its corners correspond with the standard basis vectors $e_1$, $e_2$, and $e_3$, and by definition have label 1, 2, and 3, respectively.

\section{KKM implies Sperner}

Let us start with the formulations of the KKM lemma and Sperner's lemma:

\begin{namedthm*}{The KKM lemma}
If $C_1, \ldots, C_n$ are closed subsets of $\Delta$ and for each nonempty $J \subseteq \{1, \ldots, n\}$, set $\Delta_J := \conv \{e_j: j \in J\}$ is a subset of\, $\bigcup_{j \in J} C_j$, then\, $\bigcap_{i=1}^n C_i \neq \emptyset$.
\end{namedthm*}

\begin{namedthm*}{Sperner's lemma}
Consider a simplicial subdivision of $\Delta$ and a Sperner labeling. This subdivision contains a completely labeled simplex.
\end{namedthm*}

For a direct proof that the KKM lemma implies Sperner's lemma, we find sets $C_i$ in the KKM lemma such that points in their intersection lie in a completely labeled simplex. The intuition is to exploit the unique representation of elements in a simplex as a convex combination of its vertices. If $x$ lies in a simplex and its barycentric coordinate/weight on a vertex $v$ is strictly positive, then that vertex must lie in every face containing $x$: had it been absent, $x$ would have a second representation that did not use $v$. So we take $C_i$ to be the elements of simplices in the subdivision whose weights on a vertex with label $i$ are sufficiently large --- for convenience, at least $1/n$: since simplices in the subdivision of $\Delta$ have $n$ vertices and weights sum to one, there is always a vertex with weight $1/n$ or more. Now if $x$ belongs to $C_1$ and $C_2$, it must lie in a face with labels 1 and 2. Repeating this, we find a simplex with all labels. The sets $C_1$, $C_2$, and $C_3$ for our example in $\mathbb{R}^3$ are sketched in Fig. \ref{fig: simplex}. To find $C_1$, for instance, go through all simplices of the subdivision and color all elements where a vertex with label 1 has barycentric coordinate 1/3 or more. Their union is $C_1$.

A more naive approach (yes, my first guess), to define $C_i$ as the set of points in simplices with label $i$, does not work. The starred vertex in Fig. \ref{fig: simplex} lies in a simplex of the subdivision with label $i$, no matter what label $i \in \{1, 2, 3\}$ you choose, but not in a completely labeled simplex.

\begin{figure}
\centering
\begin{tikzpicture}[scale=0.8, transform shape]
% Place vertices of the simplicial subdivision
\coordinate (e1) at (0:0);
\coordinate (e2) at (0:4);
\coordinate (e3) at (60:4);
\coordinate (a) at (barycentric cs:e1=2,e2=1);
\coordinate (b) at (barycentric cs:e1=1,e2=2);
\coordinate (c) at (barycentric cs:e2=3,e3=2);
\coordinate (d) at (barycentric cs:e1=1,e3=1);
\coordinate (e) at (barycentric cs:e1=3,e2=3,e3=3);
% Draw unit simplex
\draw[thick] (e1) node[below] {1} -- (e2) node[below] {2} -- (e3) node[above] {3} -- cycle;
% Draw subdivision
\draw (e) node[above right] {2} -- (c) node[right] {2$^{\bigstar}$} -- (b) node[below] {1} -- (e) -- (a) node[below] {2} -- (d) node[left] {1} -- (e) -- (e3);
% Draw fully labeled simplex
\draw[fill=lightgray] (e3) -- (d) -- (e) -- cycle;
\end{tikzpicture}
\hspace{3pt}
\begin{tikzpicture}[scale=0.8, transform shape]
% Place vertices of the simplicial subdivision
\coordinate (e1) at (0:0);
\coordinate (e2) at (0:4);
\coordinate (e3) at (60:4);
\coordinate (a) at (barycentric cs:e1=2,e2=1);
\coordinate (b) at (barycentric cs:e1=1,e2=2);
\coordinate (c) at (barycentric cs:e2=3,e3=2);
\coordinate (d) at (barycentric cs:e1=1,e3=1);
\coordinate (e) at (barycentric cs:e1=3,e2=3,e3=3);
% Draw unit simplex
\draw[thick] (e1) node[below] {1} -- (e2) node[below] {2} -- (e3) node[above] {3} -- cycle;
% Draw subdivision
\draw (e) node[above right] {2} -- (c) node[right] {2} -- (b) node[below] {1} -- (e) -- (a) node[below] {2} -- (d) node[left] {1} -- (e) -- (e3);
% Draw C1
\draw[thick, blue, fill opacity=0.5, fill = blue!20] (e1) -- (barycentric cs:e1=1,a=2) -- (barycentric cs:e1=1,a=1,d=1) -- (barycentric cs:d=1,a=2) -- (barycentric cs:d=1,e=2) -- (barycentric cs:d=1,e3=2) -- cycle;
\draw[thick, blue, fill opacity=0.5, fill = blue!20] (barycentric cs:b=1,a=2) -- (barycentric cs:b=1,e2=2) -- (barycentric cs:b=1,c=2) -- (barycentric cs:b=1,e=2) -- cycle;
\end{tikzpicture}
\hspace{3pt}
\begin{tikzpicture}[scale=0.8, transform shape]
% Place vertices of the simplicial subdivision
\coordinate (e1) at (0:0);
\coordinate (e2) at (0:4);
\coordinate (e3) at (60:4);
\coordinate (a) at (barycentric cs:e1=2,e2=1);
\coordinate (b) at (barycentric cs:e1=1,e2=2);
\coordinate (c) at (barycentric cs:e2=3,e3=2);
\coordinate (d) at (barycentric cs:e1=1,e3=1);
\coordinate (e) at (barycentric cs:e1=3,e2=3,e3=3);
% Draw unit simplex
\draw[thick] (e1) node[below] {1} -- (e2) node[below] {2} -- (e3) node[above] {3} -- cycle;
% Draw subdivision
\draw (e) node[above right] {2} -- (c) node[right] {2} -- (b) node[below] {1} -- (e) -- (a) node[below] {2} -- (d) node[left] {1} -- (e) -- (e3);
% Draw C2
\draw[thick, blue, fill opacity=0.5, fill = blue!20] (barycentric cs:e=1,d=2) -- (barycentric cs:e=1,d=1,a=1) -- (barycentric cs:a=1,d=2) -- (barycentric cs:a=1,e1=2) -- (barycentric cs:a=1,b=2) -- (barycentric cs:a=1,b=1,e=1) -- (barycentric cs:e=1,b=2) -- (barycentric cs:b=1,c=1,e=1) -- (barycentric cs:b=2,c=1) -- (barycentric cs:b=1,e2=1,c=1) -- (barycentric cs:b=2,e2=1) -- (e2) -- (barycentric cs:c=1,e3=2) -- (barycentric cs:e=1,c=1,e3=1) -- (barycentric cs:e=1,e3=2) -- cycle;
\end{tikzpicture}
\hspace{3pt}
\begin{tikzpicture}[scale=0.8, transform shape]
% Place vertices of the simplicial subdivision
\coordinate (e1) at (0:0);
\coordinate (e2) at (0:4);
\coordinate (e3) at (60:4);
\coordinate (a) at (barycentric cs:e1=2,e2=1);
\coordinate (b) at (barycentric cs:e1=1,e2=2);
\coordinate (c) at (barycentric cs:e2=3,e3=2);
\coordinate (d) at (barycentric cs:e1=1,e3=1);
\coordinate (e) at (barycentric cs:e1=1,e2=1,e3=1);
% Draw unit simplex
\draw[thick] (e1) node[below] {1} -- (e2) node[below] {2} -- (e3) node[above] {3} -- cycle;
% Draw subdivision
\draw (e) node[above right] {2} -- (c) node[right] {2} -- (b) node[below] {1} -- (e) -- (a) node[below] {2} -- (d) node[left] {1} -- (e) -- (e3);
% Draw C3
\draw[thick, blue, fill opacity=0.5, fill = blue!20] (e3) -- (barycentric cs:e3=1,d=2) -- (barycentric cs:e3=1,e=2) -- (barycentric cs:c=2,e3=1) -- cycle;
\end{tikzpicture}
\caption{A Sperner labeling with one completely labeled simplex and the sets $C_1$, $C_2$, and $C_3$.}\label{fig: simplex}
\end{figure}
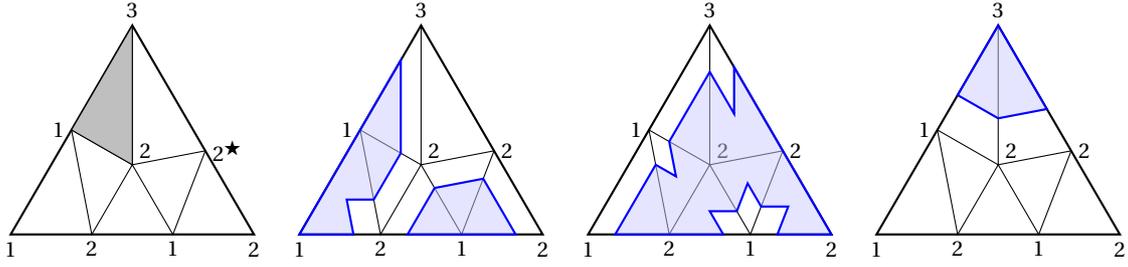

\begin{theorem}
The KKM lemma implies Sperner's lemma.
\end{theorem}

\begin{proof}
\noindent \bem{Define $C_i$.} Fix label $i \in \{1, \ldots, n\}$. For each simplex $S = \conv \{a_1, \ldots, a_n\}$ in the subdivision, define the possibly empty subset $T_S \subseteq S$ of convex combinations $\sum_j \lambda_j a_j$ giving weight $\lambda_j \geq 1/n$ to at least one vertex $a_j$ with label $i$. Let $C_i$ be the union of all these $T_S$. As the union of finitely many closed sets, $C_i$ is closed.

\noindent \bem{Verify the KKM condition.} Let $J \subseteq \{1, \ldots, n\}$ be nonempty. To show: $\Delta_J \subseteq \bigcup_{j \in J} C_j$. So let $x \in \Delta_J$. This $x$ is a convex combination $x = \sum_i \lambda_i a_i$ of vertices of a simplex $S = \conv \{a_1, \ldots, a_n\}$ of the subdivision. At least one vertex $a_i$ has weight $\lambda_i \geq 1/n$. Since $x \in \Delta_J$, only coordinates of $x$ and hence of $a_i$ that lie in $J$ can be positive. By definition of the labeling, $\ell(a_i) \in J$. So $x \in C_{\ell(a_i)} \subseteq \bigcup_{j \in J} C_j$.

\noindent \bem{Find a completely labeled simplex.} By KKM, there is an $x^* \in \bigcap_i C_i$. By definition of $C_i$: for each label $i$, some simplex $S_i$ in the subdivision has $x^*$ as a convex combination of its vertices with a weight of at least $1/n$ on a vertex $v_i$ with label $i$. $S_1$ is completely labeled. Its vertex $v_1$ has label 1.  For labels $p \neq 1$, note that $x^* \in S_1 \cap S_p$. This intersection is a face of $S_1$ and $S_p$. But then $v_p$ with label $p$ is one of the vertices (of $S_1$ and $S_p$) spanning this face: if not, $x^*$ would also be a convex combination of vertices of $S_p$ \emph{other than\/} $v_p$, contradicting its unique representation in $S_p$.
\end{proof}

\section*{Acknowledgements}

I thank J\"{o}rgen Weibull, Henrik Petri, and Albin Erlanson for helpful discussions. Financial support from the Wallander-Hedelius foundation through grant P2010-0094:1 is gratefully acknowledged.

\bibliographystyle{abbrvnat}
\bibliography{KKMimpliesSpernerReferences}
\end{document}